\documentclass{amsart}

\usepackage[english]{babel}
\usepackage{amsmath,amssymb,amsfonts, amsthm}
\usepackage[utf8]{inputenc}
\usepackage[T1]{fontenc}
\usepackage{enumerate}
\usepackage{enumitem}
\usepackage{graphicx}
\usepackage{lmodern}
\usepackage{dsfont}
\usepackage[colorlinks, citecolor=blue, linkcolor=blue]{hyperref}

\theoremstyle{plain}
\newtheorem{theorem}{Theorem}[section]
\newtheorem*{theorem*}{Theorem}

\newtheorem{lemma}[theorem]{Lemma}
\theoremstyle{definition}

\newtheorem*{definition*}{Definition}

\numberwithin{equation}{section}

\title[]{Spectral bounds for singular indefinite Sturm-Liouville operators with $L^1$--potentials}

\author{Jussi Behrndt}
\address{Institut f\"ur Numerische Mathematik,
    Technische Universit\"at Graz,
    Steyrergasse 30,
    A-8010 Graz,
    Austria}
\email{behrndt@tugraz.at}

\author{Philipp Schmitz}
\address{Institut f\"ur  Mathematik,  Technische Universit\"{a}t Ilmenau, Postfach 100565, D-98684 Ilmenau,  Germany}
\email{philipp.schmitz@tu-ilmenau.de}

\author{Carsten Trunk}
\address{Institut f\"ur  Mathematik,  Technische Universit\"{a}t Ilmenau, Postfach 100565, D-98684 Ilmenau,  Germany}
\email{carsten.trunk@tu-ilmenau.de}


\begin{document}
\begin{abstract} 
The spectrum of the
singular indefinite Sturm-Liouville operator $$A=\operatorname{sgn}(\cdot)\bigl(-\tfrac{d^2}{dx^2}+q\bigr)$$
with a real potential
$q\in L^1(\mathbb R)$ covers the whole real line and, in addition, non-real eigenvalues may appear if the potential $q$ assumes negative values. 
A quantitative analysis of the non-real eigenvalues is a 
challenging problem, and so far only partial results in this direction were obtained. 
In this paper the bound $$\vert\lambda\vert\leq \Vert q\Vert_{L^1}^2$$ on the absolute values of the non-real eigenvalues $\lambda$ of $A$ is obtained. Furthermore,
separate bounds on the imaginary parts and absolute values of these eigenvalues are proved in terms of the $L^1$-norm of the negative part of $q$.
% \\
% \\
%  
%   The real line is always contained in the spectrum of
%  a singular indefinite Sturm-Liouville operator $$A=\operatorname{sgn}(\cdot)\bigl(-\tfrac{d^2}{dx^2}+q\bigr)$$ in $L^2(\mathbb R)$ with a real potential
%  $q\in L^1(\mathbb R)$. Non-real eigenvalues  occur, only if
%  the potential $q$ assumes negative values. So far, not much is known
%  about the retaionship between $q$ and the location of the non-real
%  eigenvalues of $A$.
%  
%  Here we show how the $L^1$-norm of the potential $q$ dominates the
%  the  absolute values of the non-real eigenvalues $\lambda$ of $A$:
%  $$
%  \vert\lambda\vert\leq \Vert q\Vert_{L^1}^2
%  $$
%  Furthermore, we present
%  separate bounds on the imaginary part and the absolute value
%   of these eigenvalues in terms of the $L^1$-norm of $q$
%  and its negative part $q_-$.\\
% \\
% 
% 
% For a singular indefinite Sturm-Liouville operator $$A=\operatorname{sgn}\bigl(-\tfrac{d^2}{dx^2}+q\bigr)$$ in $L^2(\mathbb R)$ with real potential 
% $q\in L^1(\mathbb R)$ the bound $\vert\lambda\vert\leq \Vert q\Vert_{L^1}^2$ on the absolute values of the non-real eigenvalues $\lambda$ is obtained. Furthermore,
% separate bounds on the imaginary part and real part of these eigenvalues are proved in terms of the $L^1$-norm of $q$
% and its negative part $q_-$.
\end{abstract}
\maketitle

%%%%%%%%%%%%%%%%%%%%%%%%%%%%%%%%%%%%%%%%%%%%%%%%%%%%%%%%%%%%%%%%%%%%%%%%
%----------------------------------------------------------------------%
%----------------------------------------------------------------------% %----------------------------------------------------------------------%
%%%%%%%%%%%%%%%%%%%%%%%%%%%%%%%%%%%%%%%%%%%%%%%%%%%%%%%%%%%%%%%%%%%%%%%%

\section{Introduction}

The aim of this paper is to prove bounds on the absolute values of the non-real eigenvalues of the singular indefinite Sturm-Liouville operator
\begin{equation*}
\begin{split}
 Af &= \operatorname{sgn}(\cdot)\bigl(-f''+qf\bigr),\\ \operatorname{dom} A&=\bigl\{f\in L^2(\mathbb R): f,f'\in AC(\mathbb R),-f''+qf\in L^2(\mathbb R)\bigr\},
\end{split}
\end{equation*}
where $AC(\mathbb R)$ stands for space of all locally absolutely continuous
functions. It will always be assumed that the potential $q$ is real-valued and belongs to
$L^1(\mathbb R)$.

The operator $A$ is not symmetric nor self-adjoint in an $L^2$-Hilbert space
due to the sign change of the weight function $\operatorname{sgn}(\cdot)$.
However, $A$ can be interpreted as a self-adjoint operator with respect to the Krein space inner product
$(\operatorname{sgn}\cdot,\cdot)$ in $L^2(\mathbb R)$.
We summarize the qualitative spectral properties of $A$ in the next theorem, which follows from \cite[Theorem~4.2]{BehrndtPhilipp10} or
 \cite[Proposition 2.4]{KT09} and
the well-known spectral properties of the definite Sturm-Liouville operator $-\frac{d^2}{dx^2}+q$;
cf. \cite{Teschl14,Weidmann87,Weidmann03}.

\begin{theorem}\label{t1}
The essential spectrum of $A$ coincides with $\mathbb R$ and the non-real spectrum of $A$ consists of isolated eigenvalues with finite algebraic multiplicity
which are symmetric with respect to $\mathbb R$.
\end{theorem}

Indefinite Sturm-Liouville operators have been studied for more than a century, and have again attracted a lot of attention in the recent past.
Early works in this context usually deal with the regular case, that is, the operator $A$ is studied on a finite interval with appropriate
boundary conditions at the endpoints; cf. \cite{H14,R18} and, e.g., \cite{CL89,M86,Z05}. In this situation the spectrum of $A$ is purely discrete and
various estimates on the real  and imaginary parts  of the non-real
eigenvalues were obtained in the last few years; cf.\ \cite{BehrndtChenPhlippQi,ChenQi14,ChenQiXie16,GuoSunXie17,KikonkoMingarelli16,QiXie13}.
The singular case is much less studied, due to the technical difficulties which, very roughly speaking, are caused
by the presence of continuous spectrum.
%and solutions of the eigenvalue problem that are not square integrable.

Explicit bounds on non-real eigenvalues for singular Sturm-Liouville
operators with
$L^\infty$-potentials were obtained with Krein space perturbation techniques in \cite{BehrndtPhilippTrunk13} and under additional assumptions
for $L^1$-potentials in \cite{BehrndtSchmitzTrunkPAMM16,BehrndtSchmitzTrunkPAMM17}, see also \cite{BehrndtKatatbehTrunk09} for the absence of real eigenvalues
and \cite{LS16} for the accumulation of non-real eigenvalues of a very particular family of potentials.
In this paper we substantially improve the earlier bounds
in \cite{BehrndtSchmitzTrunkPAMM16,BehrndtSchmitzTrunkPAMM17}
and relax the conditions on the potential. More precisely, here we prove for arbitrary real $q\in L^1(\mathbb R)$ the following bound.

\begin{theorem}\label{ttt1}
Let $q\in L^1(\mathbb R)$ be real. Every non-real eigenvalue $\lambda$ of the
indefinite Sturm-Liouville operator $A$ satisfies
\begin{equation}\label{davisbound1}
|\lambda|\leq \|q\|_{L^1}^2.
\end{equation}
\end{theorem}

Moreover, we prove two bounds in terms of  the negative part $q_-$ of $q$.
\begin{theorem}\label{ttt2}
Let $q\in L^1(\mathbb R)$ be real. Every non-real eigenvalue $\lambda$ of the
indefinite Sturm-Liouville operator $A$ satisfies
	\begin{align}
	\label{imagpart1}
	|\operatorname{Im}\lambda|\leq 24\cdot\sqrt{3}\|q_-\|_{L^1}^2\quad\text{and}\quad |\lambda| \leq (24\cdot\sqrt{3} +18)\|q_-\|_{L^1}^2.
	\end{align}
\end{theorem}

The bound \eqref{davisbound1} is proved in Section~\ref{BirmanSchwinger}.
Its proof is based on the Birman-Schwinger principle using similar arguments as in
\cite{AbramovAslanyanDavies,DN02}, \cite[Chapter 14.3]{Davies07}; see also \cite{BE02}. The bounds in \eqref{imagpart1} are obtained in Section~\ref{sec3}
by adapting the techniques from the regular case in \cite{BehrndtChenPhlippQi,ChenQi14,QiXie13} to the present singular situation.

\section{Proof of Theorem \ref{ttt1}}\label{BirmanSchwinger}
In this section we prove the bound \eqref{davisbound1} for the non-real eigenvalues of $A$. We adapt a technique similar to the Birman-Schwinger principle in \cite{Davies07} and apply it to the indefinite operator $A$.
The main ingredient is a bound for the integral kernel of the resolvent of the operator
\begin{align*}
	B_0f=\operatorname{sgn}(\cdot)\bigl(-f''\bigr),
	\quad \operatorname{dom} B_0=\bigl\{f\in L^1(\mathbb R): f,f'\in AC(\mathbb R), -f''\in L^1(\mathbb R)\bigr\},
\end{align*}
in $L^1(\mathbb R)$.
% and the observation that an eigenfunction in $L^2(\mathbb R)$ is also an eigenfunction for the indefinite
%operator $\operatorname{sgn}(\cdot)(-\frac{d^2}{dx^2}+q)$ in $L^1(\mathbb R)$.

\begin{lemma}\label{blem}
	The operator $B_0$
	 is closed in $L^1(\mathbb R)$ and for all $\lambda$ in the open upper half-plane $\mathbb C^+$ the resolvent of $B_0$ is an integral operator
	\begin{align*}
		\big[(B_0-\lambda)^{-1}g\big](x) = \int_{\mathbb R} K_\lambda(x,y)
        g(y)\,\mathrm d y, \quad g \in L^1(\mathbb R),
	\end{align*}
	where the kernel $K_\lambda :\mathbb R \times \mathbb R \to \mathbb C$
 is bounded by $|K_\lambda(x,y)|\leq|\lambda|^{-\frac{1}{2}}$ for all
 $x,y \in \mathbb R$.
\end{lemma}
\begin{proof}
Here and in the following we define $\sqrt{\lambda}$ for $\lambda\in \mathbb C^+$
	as the principal value of the square root, which ensures $\operatorname{Im}\sqrt{\lambda}>0$ and $\operatorname{Re}\sqrt{\lambda}>0$.
	For $\lambda \in \mathbb C^+$ consider the integral operator
	\begin{align}
		\label{defT}
		(T_\lambda g)(x) = \int_{\mathbb R} K_\lambda(x,y) g(y)\,\mathrm d y,\quad g\in L^1(\mathbb R),
	\end{align}
	with the kernel $K_\lambda(x,y)=C_\lambda(x,y)+D_\lambda(x,y)$ of the form
	\begin{align*}
	C_\lambda(x,y)=\frac{1}{2\alpha\sqrt{\lambda}}\begin{cases}
	\alpha e^{i\sqrt\lambda (x+y)},& x\geq 0,\,y\geq 0,\\
	-e^{\sqrt{\lambda}(ix+y)},&x\geq 0,\,y<0,\\
	e^{\sqrt{\lambda}(x+iy)},&x<0,\,y\geq 0,\\
	-\overline{\alpha}e^{\sqrt{\lambda}(x+y)},&x<0,\,y<0,
	\end{cases}
	\end{align*}
	and
	\begin{align*}
	D_\lambda(x,y)=\frac{1}{2\alpha\sqrt{\lambda}}\begin{cases}
	\overline{\alpha} e^{i\sqrt{\lambda}|x-y|},& x\geq 0,\,y\geq 0,\\
	0,&x\geq 0,\,y<0,\\
	0,&x<0,\,y\geq 0,\\
	-\alpha e^{-\sqrt{\lambda}|x-y|},&x<0,\,y<0,
	\end{cases}
	\end{align*}
	where $\alpha:=\frac{1-i}{2}$. Hence,
	\begin{align*}
		|K_\lambda(x,y)|=|C_\lambda(x,y)+D_\lambda(x,y)|\leq \frac{1}{\sqrt{|\lambda|}}
	\end{align*}
	and the integral in \eqref{defT} converges for every $g\in L^1(\mathbb R)$.
	We have
	\begin{align*}
	\sup_{y\geq 0} \int_{\mathbb R} |C_\lambda(x,y)|\,\mathrm dx =
\frac{1}{2\sqrt{|\lambda|}}\left(\frac{1}{\operatorname{Im}\sqrt{\lambda}}
+ \frac{\sqrt{2}}{\operatorname{Re}\sqrt{\lambda}}\right)
	\end{align*}
and
\begin{align*}
	\sup_{y< 0} \int_{\mathbb R} |C_\lambda(x,y)|\,\mathrm dx =
\frac{1}{2\sqrt{|\lambda|}}\left(\frac{\sqrt{2}}{\operatorname{Im}\sqrt{\lambda}}
+ \frac{1}{\operatorname{Re}\sqrt{\lambda}}\right).
	\end{align*}
	For $y\geq 0$ we estimate
	\begin{align*}
	\int_{0}^\infty |D_\lambda(x,y)|\,\mathrm dx = \frac{1}{2\sqrt{|\lambda|}}\int_{0}^\infty e^{-\operatorname{Im}\sqrt{\lambda}|x-y|}\,
	\mathrm dx=\frac{2-e^{-\operatorname{Im}\sqrt{\lambda}y}}{2\sqrt{|\lambda|}\operatorname{Im}\sqrt{\lambda}}\leq \frac{1}{\sqrt{|\lambda|}\operatorname{Im}\sqrt{\lambda}},
	\end{align*}
	and analogously for $y<0$
	\begin{align*}
	\int_{-\infty}^0 |D_\lambda(x,y)|\,\mathrm dx = \frac{1}{2\sqrt{|\lambda|}}\int_{-\infty}^0 e^{-\operatorname{Re}\sqrt{\lambda}|x-y|}\,\mathrm dx=\frac{2-e^{\operatorname{Re}\sqrt{\lambda}y}}{2\sqrt{|\lambda|}\operatorname{Re}\sqrt{\lambda}}\leq \frac{1}{\sqrt{|\lambda|}\operatorname{Re}\sqrt{\lambda}}.
	\end{align*}
	Hence,
	\begin{align*}
	c:=\sup_{y\in\mathbb R}\int_{\mathbb R} |K_\lambda(x,y)|\,\mathrm dx <\infty
	\end{align*}
	and Fubini's theorem yields
	\begin{align*}
	\|T_\lambda g\|_{L^1} \leq \int_{\mathbb R}|g(y)|\int_{\mathbb R} |K_\lambda(x,y)|\,\mathrm dx\,\mathrm dy \leq c \|g\|_{L^1}.
	\end{align*}
	Therefore $T_\lambda$ in \eqref{defT} is an everywhere defined bounded operator in $L^1(\mathbb R)$.
	
	We claim that $T_\lambda$ is the inverse of $B_0-\lambda$. In fact,
	consider the functions $u,v$ given by
	\begin{align*}
	u(x) =\begin{cases}
	e^{i\sqrt{\lambda}x},& x\geq 0,\\ \overline{\alpha} e^{\sqrt{\lambda}x} + \alpha e^{-\sqrt{\lambda}x},& x<0,
	\end{cases}
	\quad\text{and}\quad
	v(x) = \begin{cases}
	\alpha e^{i\sqrt{\lambda}x} + \overline{\alpha} e^{-i\sqrt{\lambda}x},& x\geq 0,\\
	e^{\sqrt{\lambda}x},& x< 0,\\
	\end{cases}
	\end{align*}
which  solve the differential equation $\operatorname{sgn}(\cdot)(-f'')=\lambda f$, that is, 
$u$ and $v$, and their derivatives, belong to $AC(\mathbb R)$ and satisfy the differential
equation almost everywhere.
	Since the Wronskian equals $2\alpha \sqrt{\lambda}$, these solutions are linearly independent. Note that $u,v\notin L^1(\mathbb R)$ and one concludes that
	$B_0-\lambda$ is injective.
	A simple calculation shows the identity
	\begin{align*}
	K_\lambda(x,y)=C_\lambda(x,y) + D_\lambda(x,y)=\frac{1}{2\alpha\sqrt{\lambda}}
		\begin{cases} u(x) v(y) \operatorname{sgn}(y),  & y<x, \\ v(x) u(y)\operatorname{sgn}(y), & x<y, \end{cases}
	\end{align*}
	and hence we have
	\begin{align*}
	(T_\lambda g)(x) = \frac{1}{2\alpha\sqrt{\lambda}} \left(u(x)\int_{-\infty}^{x} \!v(y)\operatorname{sgn}(y)g(y)\,\mathrm dy
	+ v(x)\int_{x}^{\infty}\!u(y)\operatorname{sgn}(y)g(y)\,\mathrm dy\right).
	\end{align*}
	One verifies $T_\lambda g, (T_\lambda g)'\in AC(\mathbb R)$ and $T_\lambda g$ is a solution of $\operatorname{sgn}(\cdot)(-f'')-\lambda f =g$. This implies
	$(T_\lambda g)''\in L^1(\mathbb R)$ and hence  $T_\lambda g\in\operatorname{dom}B_0$ satisfies
	\begin{align*}
	(B_0-\lambda)T_\lambda g = g\quad\text{for all }g\in L^1(\mathbb R).
	\end{align*}
	Therefore, $B_0-\lambda$ is surjective and we have $T_\lambda=(B_0-\lambda)^{-1}$. It follows that $B_0$ is a closed operator
	in $L^1(\mathbb R)$ and that $\lambda$ belongs to the resolvent set of $B_0$.
\end{proof}

\begin{proof}[Proof of Theorem \ref{ttt1}]
	Since the non-real point spectrum of $A$ is symmetric with respect to the real line
	(see Theorem \ref{t1}) it suffices to consider eigenvalues in the upper half plane.
	Let $\lambda \in \mathbb C^+$ be an
	eigenvalue of $A$ with a corresponding eigenfunction $f\in\operatorname{dom} A$.
	Since $q\in L^1(\mathbb R)$ and $-\frac{\mathrm d^2}{\mathrm d x^2} + q$ is in the limit point
case at $\pm\infty$ (see, e.g. \cite[Lemma~9.37]{Teschl14}) the function $f$ is unique up to a constant multiple.
	As $-f''+qf=\lambda f$ on $\mathbb R^+$ and $f''-qf=\lambda f$ on $\mathbb R^-$
	with $q$ integrable one has the well-known asymptotical behaviour
	\begin{equation}\label{asy1}
	\begin{split}
	f(x) &= \alpha_+ \big(1+  o(1)\big) e^{i\sqrt{\lambda} x},\quad x\rightarrow +\infty,\\
	f'(x) &= \alpha_+ i\sqrt{\lambda} \big(1+  o(1)\big) e^{i\sqrt{\lambda} x}, \quad x\rightarrow +\infty,
	\end{split}
	\end{equation}
	and
	\begin{equation}\label{asy2}
	\begin{split}
	f(x) &= \alpha_- \big(1+ o(1)\big) e^{\sqrt{\lambda} x},\quad x\rightarrow -\infty,\\
	f'(x) &= \alpha_-\sqrt{\lambda} \big(1+ o(1)\big) e^{\sqrt{\lambda} x},\quad x\rightarrow -\infty,
	\end{split}
	\end{equation}
	for some $\alpha_+,\alpha_-\in\mathbb C$; see, e.g. \cite[§ 24.2, Example a]{Naimark68} or \cite[Lemma~9.37]{Teschl14}.
	These asymptotics yield $f,qf\in L^1(\mathbb R)$ and $-f'' = \lambda \operatorname{sgn}(\cdot)f-qf\in L^1(\mathbb R)$,
	and therefore $f\in \operatorname{dom}B_0$.
	Thus, $f$ satisfies
	\begin{align*}
		0=(A-\lambda)f=\operatorname{sgn}(\cdot)(-f'') - \lambda f + \operatorname{sgn}(\cdot)qf=(B_0-\lambda)f +\operatorname{sgn}(\cdot)qf
	\end{align*}
	and since $\lambda$ is in the resolvent set of $B_0$ we obtain
	\begin{align*}
		-qf = q(B_0-\lambda)^{-1}\operatorname{sgn}(\cdot)qf.
	\end{align*}
	Note that $\|qf\|_{L^1}\not=0$ as otherwise $\lambda$ would be an eigenvalue of $B_0$.
	With the help of Lemma~\ref{blem} we then conclude
	\begin{align*}
		0<\|qf\|_{L^1} \leq\int_{\mathbb R} |q(x)| \int_{\mathbb R} |K_\lambda(x,y)|
		|q(y)f(y)|\,\mathrm d y\,\mathrm d x \leq \frac{1}{\sqrt{|\lambda|}}\|qf\|_{L^1} \|q\|_{L^1}
	\end{align*}
	and this yields the desired bound \eqref{davisbound1}.
\end{proof}

\section{Proof of Theorem \ref{ttt2}}\label{sec3}

In this section we prove the bounds in \eqref{imagpart1} for the non-real eigenvalues of $A$ in Theorem~\ref{ttt2},
which depend only on the negative part $q_-(x)=\max\{0,-q(x)\}$, $x\in \mathbb R$, of the potential.
The following lemma will be useful.

\begin{lemma}
	\label{chinaLemma}
	Let $\lambda\in\mathbb C^+ $ be an eigenvalue of $A$ and let $f$ be a corresponding eigenfunction. Define
	\begin{align*}
	U(x):=\int_{x}^\infty \operatorname{sgn}(t)|f(t)|^2\,\mathrm dt\quad\text{and}\quad V(x):= \int_x^\infty |f'(t)|^2 + q(t) |f(t)|^2\,\mathrm dt.
	\end{align*}
	for $x\in\mathbb R$.
	Then the following assertions hold:
	\begin{itemize}
        \setlength\itemsep{1ex}
		\item[{\rm (a)}] $\lambda U(x) = f'(x)\overline{f(x)} + V(x)$;
		\item[{\rm (b)}] $\lim_{x\rightarrow -\infty}U(x)=0$ and $\lim_{x\rightarrow -\infty}V(x)=0$;
		\item[{\rm (c)}] $\|f'\|_{L^2}\leq 2\|q_-\|_{L^1} \|f\|_{L^2}$;
		\item[{\rm (d)}] $\|f\|_{\infty}\leq 2\sqrt{\|q_-\|_{L^1}} \|f\|_{L^2}$;
        \item[{\rm (e)}] $\|qf^2\|_{L^1}\leq 8\|q_-\|_{L^1}^2 \|f\|_{L^2}^2$.
	\end{itemize}
\end{lemma}

\begin{proof}
Note that $f$ satisfies the asymptotics \eqref{asy1}--\eqref{asy2} and hence
$f$ and $f'$ vanish at $\pm\infty$ and $f'\in L^2(\mathbb R)$. In particular, $V(x)$ is well defined.
We multiply the identity $\lambda f(t)= \operatorname{sgn}(t)(-f''(t)+q(t) f(t))$ by $\operatorname{sgn}(t)\overline{f(t)}$ and integration by parts yields  
	\begin{align*}
	%\label{passau}
	\lambda U(x) = \int_x^\infty -f''(t)\overline{f(t)} + q(t)|f(t)|^2 \,\mathrm dt =  f'(x)\overline{f(x)} + V(x)
	\end{align*}
	for all $x\in\mathbb R$. This shows (a). Moreover, we have 
	\begin{align*}
	\lambda \int_{\mathbb R} \operatorname{sgn}(t)|f(t)|^2\,\mathrm dt = 
\lim_{x\to -\infty} \lambda U(x) = 
\lim_{x\to -\infty}V(x) = 
\int_{\mathbb R} |f'(t)|^2+ q(t) |f(t)|^2\,\mathrm dt.
	\end{align*}
Taking the imaginary part shows $\lim_{x\to -\infty}U(x)=0$ and, hence,
$\lim_{x\to -\infty}V(x)=0$. This proves (b).

	As $f$ is continuous and vanishes at $\pm\infty$ we have $\|f\|_{\infty}<\infty$. Let $q_+(x) := \max\{0,q(x)\}$, $x\in\mathbb R$. Making use of $\lim_{x\rightarrow -\infty}V(x)=0$ and $q=q_+-q_-$ we find
	\begin{align}
	\label{emb1}
    \begin{split}
	0\leq\|f'\|_{L^2}^2& = -\int_{\mathbb R}q(t)|f(t)|^2\,\mathrm dt = -\int_{\mathbb R}\big(q_+(t)-q_-(t)\big)|f(t)|^2\,\mathrm dt\\
     &\leq \int_{\mathbb R} q_-(t)|f(t)|^2\,\mathrm dt
	\leq \|q_-\|_{L^1}\|f\|_\infty^2.
    \end{split}
	\end{align}
    This implies $\|q_+ f^2\|_{L^1}\leq \|q_-f^2\|_{L^1}\leq \|q_-\|_{L^1}\|f\|_\infty^2$ and, thus,
    \begin{align}
    \label{estq}
    \|qf^2\|_{L^1}=\int_{\mathbb R} |q(t)| |f(t)|^2\,\mathrm dt=\int_{\mathbb R} \big(q_+(t)+q_-(t)\big) |f(t)|^2\,\mathrm dt\leq 2\|q_-\|_{L^1}\|f\|_\infty^2.
    \end{align}
	 In order to verify (d) let $x,y\in\mathbb R$ with $x>y$. Then
	\begin{align*}
	|f(x)|^2-|f(y)|^2 = \int_{y}^{x} \bigl(|f|^2\bigr)'(t)\,\mathrm dt \leq 2\int_{y}^x |f(t)f'(t)|\,\mathrm dt\leq 2\|f\|_{L^2} \|f'\|_{L^2}	
	\end{align*}
	together with $f(y)\rightarrow 0$, $y\rightarrow -\infty$, leads to  $\|f\|^2_\infty\leq 2 \|f\|_{L^2}\|f'\|_{L^2}$. Since $f$ is an eigenfunction $\|f\|_\infty$ does not vanish and we have with \eqref{emb1}
    \begin{align*}
        \|f\|_\infty \leq \frac{2 \|f\|_{L^2}\|f'\|_{L^2}}{\|f\|_\infty} \leq 2\sqrt{\|q_-\|_{L^1}} \|f\|_{L^2},
    \end{align*}
    which shows (d). Moreover, the estimate in (d) applied to \eqref{emb1} and \eqref{estq} yield (c) and (e).
\end{proof}

\begin{proof}[Proof of Theorem~\ref{ttt2}]
	Let $\lambda \in \mathbb C^+$ be a eigenvalue of $A$ and
	let $f\in\operatorname{dom}A$ be a corresponding eigenfunction. We can assume $\|q_-\|_{L^1}>0$ as otherwise $f=0$ by
	Lemma~\ref{chinaLemma} (d). Let $U$ and $V$ be as in Lemma~\ref{chinaLemma},
	let $\delta:= (24 \|q_-\|_{L^1})^{-1}$ and define the function $g$ on  $\mathbb R$ by
	\begin{align*}
	g(x)=\begin{cases}
	\operatorname{sgn}(x),& |x|>\delta,\\
	\frac{x}{\delta}, & |x|\leq \delta.
	\end{cases}
	\end{align*}
	From Lemma~\ref{chinaLemma}~(a) we have
	\begin{align}
	\label{moskau}
	\lambda\int_{\mathbb R} g'(x) U(x)\,\mathrm dx = \int_{\mathbb R} g'(x)\bigl(f'(x)\overline{f(x)} + V(x)\bigr)\,\mathrm dx.
	\end{align}
	Since $g$ is bounded and $U(x)$ vanishes for $x\rightarrow \pm\infty$, integration by parts leads to the estimate
	\begin{equation}\label{zweidrittel}\begin{split}
	\int_{\mathbb R} g'(x) U(x) \,\mathrm dx =& \int_{\mathbb R}g(x) \operatorname{sgn}(x)|f(x)|^2\,\mathrm dx
	\geq \int_{\mathbb R\setminus [-\delta,\delta]} |f(x)|^2\,\mathrm dx\\
	= & \|f\|_{L^2}^2 - \int_{-\delta}^\delta |f(x)|^2\,\mathrm dx \geq  \|f\|_{L^2}^2 -2\delta\|f\|_{\infty}^2\\
	\geq & \|f\|_{L^2}^2 - 8 \delta \|q_-\|_{L^1} \|f\|_{L^2}^2 = \frac{2}{3}\|f\|_{L^2}^2;
	\end{split}\end{equation}
	here we have used  Lemma~\ref{chinaLemma}~(d) in the last line of \eqref{zweidrittel}.
	Further we see with Lemma~\ref{chinaLemma} (c)--(d) 
	\begin{equation}\label{colfo}\begin{split}
	\left|\int_{\mathbb R} g'(x)f'(x)\overline{f(x)}\,\mathrm dx \right|&\leq \|f\|_\infty \|f'\|_{L^2} \|g'\|_{L^2}\leq 4 \|q_-\|_{L^1}^{\frac{3}{2}} \|f\|_{L^2}^2 \sqrt{\frac{2}{\delta}}\\
	&\leq 16\cdot \sqrt{3}\|q_-\|_{L^1}^{2} \|f\|_{L^2}^2.
	\end{split}\end{equation}
	Since $\Vert g\Vert_\infty=1$ and $V(x)$ vanishes for $x\rightarrow\pm\infty$ integration by parts together with Lemma~\ref{chinaLemma}~(c) and (e) yields
	\begin{equation}\label{colfo2}\begin{split}
	\left|\int_{\mathbb R} g'(x) V(x)\,\mathrm dx\right| &={ } \left|\int_{\mathbb R} g(x) \left(|f'(x)|^2+q(x)|f(x)|^2\right)\,\mathrm dx\right|\\
	&\leq{ } \|g\|_\infty \left(\|f'\|_{L^2}^2 + \|qf^2\|_{L^1}\right) \leq 12 \|q_-\|_{L^1}^2\|f\|_{L^2}^2.
	\end{split}\end{equation}	
	Comparing the imaginary parts in \eqref{moskau} we have with \eqref{zweidrittel} and \eqref{colfo}
	\begin{align*}
	\frac{2}{3}|\operatorname{Im}\lambda| \|f\|_{L^2}^2\leq&{ }|\operatorname{Im}\lambda|\left|\int_{\mathbb R} g'(x) U(x)\,\mathrm dx\right|
	\leq\left|\int_{\mathbb R} g'(x)f'(x)\overline{f(x)}\,\mathrm dx\right|\\
	\leq&{ }16\cdot \sqrt{3}\|q_-\|_{L^1}^{2} \|f\|_{L^2}^2.
	\end{align*}
	In the same way we obtain from \eqref{zweidrittel}, \eqref{moskau} and \eqref{colfo}--\eqref{colfo2} that 
	\begin{align*}
	\frac{2}{3}|\lambda| \|f\|_{L^2}^2\leq&{ }\left|\lambda \int_{\mathbb R} g'(x) U(x)\,\mathrm dx\right|
	=\left|\int_{\mathbb R} g'(x)\bigl(f'(x)\overline{f(x)}+ V(x)\bigr)\,\mathrm dx\right|\\
	\leq&{ }\left(16\cdot \sqrt{3}+12\right)\|q_-\|_{L^1}^{2} \|f\|_{L^2}^2.
	\end{align*}
	This shows the bounds in \eqref{imagpart1}.
\end{proof}

\end{document}